\documentclass[12pt]{amsart}
\date{Jan.~11, 2010 (revision)}
\usepackage{color}
\usepackage[all]{xy}
\usepackage{amssymb,amsmath}
\usepackage{mathpazo}
\usepackage[T1]{fontenc}
\usepackage{mathrsfs}
\usepackage{eucal}
\usepackage{textcomp}

\usepackage{setspace}

\calclayout
\newtheorem{dummy}{anything}[section]
\newtheorem{theorem}[dummy]{Theorem}

\newtheorem{lemma}[dummy]{Lemma}

\newtheorem{corollary}[dummy]{Corollary}

\theoremstyle{definition}
\newtheorem{definition}[dummy]{Definition}
  \newtheorem{example}[dummy]{Example}
  
  \newtheorem{remark}[dummy]{Remark}

\newcommand
{\eqncount}{\setcounter{equation}{\value{dummy}}%
\addtocounter{dummy}{1}}
\newcommand{\cA}{\mathcal A}
\newcommand{\cB}{\mathscr B}

\newcommand{\cE}{\mathcal E}
\newcommand{\cF}{\mathscr F}
\newcommand{\cG}{\mathscr G}

\newcommand{\bC}{\mathbf C}
\newcommand{\bD}{\mathbf D}

\newcommand{\bZ}{\mathbb Z}

\newcommand{\bP}{\mathbf P}

\newcommand{\bbQ}{\mathbb Q}

\DeclareMathOperator{\Hom}{Hom}
\DeclareMathOperator{\Irr}{Irr}

 \DeclareMathOperator{\Ext}{Ext}

 \DeclareMathOperator{\Ind}{Ind}
\DeclareMathOperator{\Res}{Res} 
 
\DeclareMathOperator{\Mor}{Mor} 
\DeclareMathOperator{\supp}{supp}

 \DeclareMathOperator{\ind}{Ind}

 \DeclareMathOperator{\Obs}{Obs}

 \DeclareMathOperator{\Lin}{Lin}

\newcommand{\cy}[1]{\bZ/{#1}}

\newcommand{\disjointunion}{\sqcup}

\newcommand{\vv}{\, | \,}

\def\G{\varGamma}

\DeclareMathOperator{\Is}{Iso}

\DeclareMathOperator\Mod{mod}
\newcommand{\leftexp}[2]{{\vphantom{#2}}^{ #1}{\hskip-1pt#2}}

\DeclareMathOperator{\Or}{Or}

\newcommand\RG{R\G}
\newcommand\ZG{\bZ\G}

\newcommand\OrG{\Or _{\cF}G}

\newcommand\un{\underline}
\newcommand\uR[1]{R[{#1}^{\, \textbf{?}\,}]}

\newcommand\uZ[1]{\mathbb{Z}[{#1}^{\, \textbf{?}\,}]}

\newcommand\uC[1]{\CC({#1}^{\textbf{?}};R)}

\newcommand\CC{\bC}
\newcommand\DD{\bD}
\newcommand\PP{\bP}

\newcommand{\nor}{\unlhd}

\def\Zphat{\widehat \bZ_p}

\begin{document}

\title{Acyclic Chain Complexes over the Orbit Category}
\author{Ian Hambleton}
\author{Erg\"un Yal\c c\i n}

\address{Department of Mathematics, McMaster University,
Hamilton, Ontario L8S 4K1, Canada}

\email{hambleton@mcmaster.ca }

\address{Department of Mathematics, Bilkent University,
06800 Bilkent, Ankara, Turkey}

\email{yalcine@fen.bilkent.edu.tr }

\thanks{Research partially supported by NSERC Discovery Grant A4000.
The first author would like  to thank the Max Planck Institut f\"ur Mathematik in Bonn for its hospitality and support while part of this work was done. 
The second author is partially supported by 
T\" UBA-GEB\. IP/2005-16.}

\begin{abstract}
Chain complexes of finitely generated free modules
over orbit categories provide natural algebraic models for finite
$G$-CW-complexes with prescribed isotropy. We prove a $p$-hypoelementary Dress induction theorem for $K$-theory over the orbit category, and use it to re-interpret some results of Oliver and Kropholler-Wall on acyclic complexes.
\end{abstract}

\maketitle

\section{Introduction}
\label{sect: introduction}

A good algebraic setting for studying actions of a group $G$ with
isotropy in a given family of subgroups $\cF$ is provided by the
category of $R$-modules  over the orbit category $\G_G=\OrG$, where
$R$ is a commutative ring with unit. This theory was established by
Bredon \cite{bredon2},  tom Dieck \cite{tomDieck2} and L\"uck
\cite{lueck3}, and further developed by many authors (see, for
example, Jackowski-McClure-Oliver \cite[\S
5]{jackowski-mcclure-oliver2}, Brady-Leary-Nucinkis
\cite{brady-leary-nucinkis1}, Symonds \cite{symonds3},
\cite{symonds2}).

The category of $\RG_G$-modules is an abelian category with $\Hom$ and
tensor product, and has enough projectives for standard homological
algebra. In this paper, we will use projective
chain complexes over the orbit category of a finite group to study acyclic
$G$-CW complexes. In Section 2 we give an orbit category version of an induction result of Dress  \cite{dress1}. In Sections \ref{sect: Oliver} and
\ref{sect: Acyclic} we re-interpret some results of Oliver
\cite{oliver0} and Kropholler-Wall \cite{kropholler-wall1} in terms
of algebra over the orbit category.

\section{Dress induction over the orbit category}\label{sect: Dress}
Let $G$ be a finite group and let $R= \Zphat$ or $R=\cy{p}$, for some prime $p$. We note that the
Krull-Schmidt theorem holds for finitely-generated $RG$-modules. Let $A(RG)$ denote the Grothendieck ring of isomorphism classes of finitely-generated $R$-torsion free
$RG$-modules,  with addition given by direct sums and product given by tensor product $\otimes_R$. By the Krull-Schmidt theorem, $A(RG)$ is  $\bZ$-torsion free.

Andreas Dress \cite[Theorem 7]{dress1} proved that  $A(RG)$  is rationally generated by induction from all the $p$-hypoelementary subgroups of $G$, and detected by restriction to the same collection of subgroups (see also Bouc \cite[Cor.~3.5.8]{bouc6} for an exposition). Recall that a subgroup $H \leq G$ is called $p$-\emph{hypoelementary} if
it has a normal $p$-subgroup $P \nor H$ such that $H/P$ is cyclic of order
 prime to $p$. 
We denote the class of $p$-hypoelementary group by
$ \cG_p^1 $. In this section we will give a version of the result of Dress for modules over the orbit category.

Let $\G_G$ denote the orbit category of $G$ with respect to  the family $\cF$ of $p$-subgroups in $G$.  Free $\RG_G$-modules are direct sums of the modules $\uR{G/Q}$, for $Q \in \cF$, where $$\uR{G/Q}(G/V) = R\Mor_G(G/V, G/Q),$$ 
and projectives are defined as direct summands of free modules. 
We will assume that the reader is somewhat familiar with modules over the orbit category (see \cite[\S 9]{lueck3}). 

In particular, we will need to use two pairs of adjoint functors
$(S_Q, I_Q)$ and $(E_Q, \Res_Q)$, defined for any object $G/Q \in \G_G$, which relate the category of right $\RG_G$-modules and the category of
right $R[N_G(Q)/Q]$-modules (see \cite[9.26-9.29]{lueck3}).  For any right $\RG_G$-module $M$, the \emph{restriction} functor is defined by
$\Res_Q(M) = M(Q)$, and the \emph{splitting} functor is given by
 $$S_Q(M) = M(Q)/M(Q)_s$$ where $M(Q)_s$ is the R-submodule generated by the images of all the $R$-homomorphisms $M(f)\colon M(K) \to M(Q)$ induced by $G$-maps $f\colon G/Q \to G/K$, with $Q <K \in \cF$. 
 For any right $R[N_G(Q)/Q]$-module $N$,
  the \emph{extension} functor 
  $$E_Q(N) = N \otimes_{R[N_G(Q)/Q]} \uR{G/Q}$$
  and the \emph{inclusion} functor is given by requiring $\Res_K(I_Q(N)) = 0$ unless $K$ and  $Q$ are conjugate, and $\Res_Q(I_Q(N)) = N$. 

The Grothendieck group of finitely-generated projective $\RG_G$-modules is denoted $K_0(\RG_G)$ (see \cite[\S 10]{lueck3} for the definition and properties of this $K_0$ functor). We remark that $K_0(\RG_G)$ is a Mackey functor under the natural operations of induction $\Ind_H^G$ and restriction $\Res_H^G$, with respect to subgroups $H \leq G$.

Let $R\cE_G$ denote the exact category of finitely-generated $R$-torsion free $\RG_G$-modules, of finite projective length over $\RG_G$, with exactness structure given by the short exact sequences of $\RG_G$-modules. 
\begin{example}\label{ex: projective length}
 Every $\RG_G$-module of the form $\uR{G/H}$, $H \leq G$,  admits a finite length projective resolution. This follows from the orbit category version of Rim's
theorem (see  \cite[Theorem 3.8]{hpy1}). However, the adjunction formula 
\cite[17.21]{lueck3} shows that, for example, the module $E_1(R) = I_1(R)$ does not have a finite length projective resolution if $G=\cy{p}$.
\end{example}

We note that $K_0(R\cE_G)$ is a ring under the operations of direct sum and tensor product $\otimes_R$, with unit $\un{R} = \uR{G/G}$ the constant $\RG_G$-module. Moreover,  $K_0(R\cE_G)$ also has the structure of a Mackey functor with respect to
$\Ind_H^G$ and $\Res_H^G$, and hence is a Green ring (via the product formulas of \cite[10.26]{lueck3}, and the observation that the diagonal functor $\Delta \colon \G_G  \to \G_G\times \G_G$ is admissible \cite[p.~203]{lueck3}). The natural map
$K_0(R\G_G) \to K_0(R\cE_G)$, sending $[P]\mapsto [P]$, is called the Cartan map.
\begin{lemma}[Grothendieck, Swan]\label{lem: green ring}
 The Cartan map 
$K_0(R\G_G) \xrightarrow{\approx} K_0(R\cE_G)$ is an isomorphism of Mackey functors.
\end{lemma}
\begin{proof}
If $M$ is an $\RG_G$-module and $\PP_* \to M$ is a projective resolution, we may define
$\chi\colon K_0(R\cE_G) \to K_0(R\G_G)$
by 
$$\chi(M)=\sum   (-1) ^i [P_i]\in 
K_0 (\RG_G). $$
The Cartan map $K_0(R\G_G) \to K_0(R\cE_G)$  is compatible with induction and restriction, and $\chi$ gives an inverse map as in Swan \cite[Thm.~1.1]{swan2}, or Curtis-Reiner  \cite[38.50]{curtis-reiner2}.
\end{proof}
\begin{lemma} $K_0(R\G_G)$ and 
$\widetilde K_0(R\G_G)$ are $\bZ$-torsion-free (for the orbit category with respect to any family $\cF$ of subgroups).
\end{lemma}
\begin{proof}
There is a (split) short exact sequence (see \cite[10.42]{lueck3})
$$0\to K_0^f(R\G_G) \to K_0(R\G_G)\to \widetilde K_0(R\G_G)\to 0$$
where $K_0^f(R\G_G)$ denotes $K_0$ of the exact category of finitely-generated free $R\G_G$-modules. In addition, there is a natural isomorphism (see L\"uck \cite[10.34]{lueck3}):
 $$K_0(\RG_G) \cong \bigoplus_{[Q]\in\Is(\G_G)} K_0(R[N_G(Q)/Q])$$
 induced by the inverse functors $S=(S_Q)$ and $E=(E_Q)$.  Here $\Is(\G_G)$ denotes the isomorphism classes of objects in $\G_G$, or equivalently the $G$-conjugacy classes of subgroups $Q \in \cF$. 
 By the Krull-Schmidt theorem, all of the groups $K_0(R[N_G(Q)/Q])$ are $\bZ$-torsion free.
\end{proof}

Here is the main result of this section.
\begin{theorem}\label{thm: computable} 
Let $\G_G$ denote the orbit category of a finite group $G$ with respect to the family of $p$-subgroups, for some prime $p$, and let $R = \Zphat$ or $R = \cy{p}$. Then $K_0 (\RG _G )\otimes \bbQ$ and $\widetilde K_0 (\RG _G )\otimes \bbQ$ are computable from the $p$-hypoelementary subgroups of $G$.
\end{theorem}
Since $\widetilde K_0(\RG_G)$ is $\bZ$-torsion free, we have the immediate consequence:
\begin{corollary}\label{cor: hypo detection} $K_0^f(R\G_G)$, $K_0 (\RG _G )$ and
$\widetilde K_0 (\RG _G )$ are detected by restriction to the sum of $\widetilde K_0 (\RG _H )$, for all  $H \in\cG_p^1$.
\end{corollary}

For $p$ and $q$ primes, let $\cG_p^q$ denote the class of  finite groups which have a normal subgroup $H \in \cG_p^1$, with $q$-power order quotient group. 
Let $\cG_p = \bigcup_q \cG_p^q$  (see Dress \cite[\S 9]{dress1} and Oliver \cite{oliver0}). 
\begin{corollary} \label{cor: hypo computation} $K_0^f(R\G_G)$, $K_0 (\RG _G )$ and
$\widetilde K_0 (\RG _G )$ are computable by induction or restriction from the family of subgroups in $\cG_p$.
\end{corollary}
\begin{proof}  Note that $\cG_p^q = hyper_q\text{-}\cG_p^1$ in the terminology of Dress, so the result follows from Corollary \ref{cor: hypo detection} and Dress induction \cite[p.~207]{dress1},  \cite[3.3]{htw2007}.
\end{proof}

\begin{proof}[The proof of Theorem \ref{thm: computable}]
The Burnside quotient Green ring $\cA_K$ of $K_0(R\cE_G)$ is isomorphic to the subring generated by the modules $\uR{G/H}$, for all $H\leq G$ (see \cite[Remark 2.4]{htw2007}). By Lemma \ref{lem: green ring} it follows that $\cA_K$ is also the Burnside quotient Green ring of the Mackey functor $K_0(\RG_G)$. Since $\widetilde K_0(\RG_G)$ is a quotient Mackey functor of $K_0(\RG_G)$, it is also a Green module over $\cA_K$ (see \cite[\S 2D]{htw2007}). By Dress induction \cite{dress2},  it suffices to show that $\cA_K\otimes \bbQ$ is generated by induction from the family of $p$-hypoelementary subgroups of $G$.

For each subgroup $H \leq G$, there is a covariant functor $F\colon \G_H \to \G_G$ of orbit categories with respect to $\cF$, such that $\Ind_F = \Ind_H^G$ and $\Res_F = \Res_H^G$ on $K$-theory. 
By  \cite[10.34]{lueck3} there is a commutative diagram
 \eqncount
\begin{equation}\label{detection1}
\vcenter{ \xymatrix@C+15pt@R+5pt{K_0(R\G_H)\ar[r]^{\Ind^G_H}& K_0(R\G_G) \ar[d]^{S(R\G_G)}_{\approx} \cr
\bigoplus\limits_{[V]\in \Is(\G_H)} K_0(R[N_H(V)/V]) \ar[r]^(0.5){F_*}\ar[u]^{E(R\G_H)}_{\approx}&\ \bigoplus\limits_{[Q]\in \Is(\G_G)} K_0(R[N_G(Q)/Q])}}
 \end{equation}
 where the  vertical maps are the splitting isomorphisms, and the lower horizontal map $F_*$ is the sum of the induction maps corresponding to the subgroups $N_H(V) \subset N_G(V)$, for $V \leq H$, and $V \in \cF$ (see \cite[10.12]{lueck3}). There is a similar diagram for $\Res_H^G$ and the contravariant map $F^*$ using \cite[10.15]{lueck3}, but the formula for $F^*$ is more complicated. The functors $E$ and $S$ are inverse pairs of natural equivalences, and we have the formulas
 $$F_* = S(R\G_G) \circ \Ind_H^G\circ E(R\G_H)$$
 and
 $$F^* = S(R\G_H) \circ \Res_H^G\circ E(R\G_G)$$
for the induced maps in diagram (\ref{detection1}). The component of $F_*$ at $[Q]\in \Is(\G_G)$ will be denoted $p_QF_*$, and similarly $p_VF^*$ will denote the component of $F^*$ at $[V] \in \Is(\G_H)$.

We wish to show that there exist rational numbers $ \{r_H\vv H \in \cG_p^1\}$ such that 
\eqncount
\begin{equation}\label{dress_formula}
a = \sum_{H \in \cG_p^1} r_H \Ind_H^G(\Res_H^G(a))
\end{equation}
for any $a\in K_0(R\G_G)$. This is equivalent to the statement that $\cA_K\otimes \bbQ$ is generated by induction from the family of $p$-hypoelementary subgroups of $G$.

We will establish formula (\ref{dress_formula}) by induction on the support 
$$\supp(a):= \{[K] \in \Is(\G_G) \vv S_K(a) \neq 0\}$$
 of an element $a \in  K_0(R[N_G(Q)/Q])$, where the support sets are partially ordered by conjugation-inclusion. 
 Since $E(\RG_G)$ is an isomorphism, we may assume that $a = E_Q(a_Q)$, for some $a_Q \in K_0(R[N_G(Q)/Q])$. Let us also assume that formula (\ref{dress_formula}) holds for all elements $b$ with $\supp(b) < \supp(a) = \{ K \leq_G Q\}$.

From the expressions above for $F_*$ and $F^*$ we have the relation
\eqncount
\begin{equation}\label{map_relation}
S(R\G_G)\left (\Ind_H^G(\Res_H^G(a))\right ) = F_*(F^*(S(R\G_G)(a))) = F_*(F^*(a_Q))
\end{equation}
so we need to compute $p_KF_*(F^*(a_Q))$, for all subgroups $K \in \cF$.

First, we compute $p_QF_*(F^*(a_Q))$, for  $a_Q  \in K_0(R[N_G(Q)/Q])$. By \cite[10.12]{lueck3}, the only non-zero components of $p_QF_*$ are given by the images of
$$\Ind_{W_H(K)}^{W_G(Q)} = \left (S_Q\circ \ind_F\circ E_K \right )_*$$
corresponding to the objects $H/K$ in $\G_H$ such that $G/Q= F(H/K)$ in $\G_G$. In other words, we need to consider only the subgroups $K \leq H$ such that $K$ is a $G$-conjugate of $Q$. Without loss of generality, we may assume that $Q \leq H$ since  $p_QF_*(F^*(a_Q)) = 0$ unless $Q$ is conjugate to a subgroup of $H$.

Therefore, we only need to consider the components $p_KF^*$ of $F^*$ which have the form
$$M \mapsto M \otimes_{R[W_G(Q)]} R\Hom_G(F(H/K), G/Q)$$
where $G/Q=F(H/K)$ and $M$ is a right $R[W_G(Q)]$-module, as given in \cite[10.15]{lueck3}. We are using the formula
$$R\Irr(H/K, G/Q) = S_K\left (R\Hom_G(F(?), G/Q)\right ) = R\Hom_G(F(H/K), G/Q)\ .$$
The right-hand side is a right 
$R[W_G(K)]$-module through the natural action of $R[W_H(K)]$ on $H/K$. But since $\Hom_G(F(H/K), G/Q) = W_G(Q)$ whenever $F(H/K) = G/Q$, each of these components of $F^*$ is just the usual restriction
$ \Res_{W_H(K)}^{W_G(K)}$ composed with a conjugation-induced isomorphism $W_G(K) \cong W_G(Q)$.

It follows that  $p_QF_*( F^*(a_Q))$ is a sum of terms indexed by the $H$-conjugacy classes of subgroups $K \leq H$ such that $K$ is $G$-conjugate to $Q$. 
We have the formula
\eqncount
\begin{equation}\label{F_composite}
p_QF_*(F^*(a_Q)) =\sum_{\{K\leq H,\,  K^g = Q\}_H}\Ind_{W_{H^g}(Q)}^{W_G(Q)} \left (  \Res_{W_{H^g}(Q)}^{W_G(Q)}(a_Q)\right ),
\end{equation}
where each term in the sum is obtained by (i) choosing an $H$-conjugacy class representative $K \leq H$, and then (ii) picking an element $g \in G$ with $K^g=Q$.

Note that the individual terms on the right-hand side of formula (\ref{F_composite}) are independent of the choices made: $K^{g_1} = K^{g_2} = Q$ implies that $W_{H^{g_1}}(Q)$ and $W_{H^{g_2}}(Q)$ are conjugate in $W_{G}(Q)$, and hence the composite $\Ind \circ \Res$ does not change.
Let $$n_{H,Q} = |\{K\leq H,\,  K^g = Q\}_H|$$ denote the number of terms in the sum (\ref{F_composite}). Alternately, $n_{H,Q}$ is the number of $N_G(Q)$-orbits in the set $(G/H)^Q$.

Similarly, the definitions of $F_*$ and $F^*$  imply that $p_KF_*(F^*(a_Q)) = 0$, unless $K\leq H$ is $G$-conjugate to a subgroup of $Q$. Therefore
$$\supp\left (E(\RG_G) (F_*(F^*(a_Q))\right ) \subseteq \supp(E_Q(a_Q)) = \supp(a).$$

By the Dress hypoelementary induction  theorem  \cite[Theorem 7]{dress1},  \cite[Cor.~3.5.8]{bouc6},  there exist rational numbers $\{t_H\vv H \in \cG_p^1\}$, such that every element $u \in A(R[N_G(Q)/Q])$ satisfies the equation
   \eqncount
\begin{equation}\label{dress:detection}
u = \sum_{H \in \cG_p^1} t_H \Ind_{W_H(Q)}^{W_G(Q)} \left (  \Res_{W_H(Q)}^{W_G(Q)}(u)\right )
\end{equation}
  where $t_H=0$ unless $Q\leq H$,   and $W_G(Q) = N_G(Q)/Q$ as usual.  We will only need the induction result for elements in the subgroup $K_0(R[N_G(Q)/Q]) \subset A(R[N_G(Q)/Q])$.  The proof shows that the general formula follows from the one for $u = [R]$, the unit in the  Dress ring (compare \cite[Theorem 3.10]{htw2007}).  We observe that the Dress formula only uses the $p$-hypoelementary subgroups of $N_G(Q)/Q$, and since $Q \in \cF$ any such subgroup has the form $H/Q$, where $Q\triangleleft H$ and $H \in \cG_p^1$. 
  
  Moreover, we can assume that the coefficients $\{t_H\}$ in (\ref{dress:detection}) are invariant under conjugation, meaning that $t_{H^g} = t_H$ for all $g\in G$. This follows by starting with the Dress induction formula for the unit $[R] \in A(R[G])$,  where the inductions and restrictions from conjugate subgroups are equal, and then obtaining the formula for $A(R[N_G(Q)/Q])$ by restriction to $R[N_G(Q)]$, followed by a generalized restriction to $R[N_G(Q)/Q]$ in the sense of 
  \cite[1.A.8]{htw3}.

We now define 
$$b = a - \sum_{H \in \cG_p^1}  \frac{t_H}{n_{H,Q}} \Ind_H^G(\Res_H^G(a))$$
for
 $a= E_Q(a_Q)\in K_0(R\G_G)$, and note that
$$  S(R\G_G) \Bigg ( \sum_{H \in \cG_p^1}  \frac{t_H}{n_{H,Q}} \Ind_H^G(\Res_H^G(a))\Bigg ) = \sum_{H \in \cG_p^1} \frac{t_H}{n_{H,Q}} F_*(F^*(S(R\G_G)(a))) $$
by formula (\ref{map_relation}). However, 
by formulas (\ref{F_composite}) and (\ref{dress:detection}) applied to $u=a_Q$,  we have 

$$S_Q(b) = S_Q(a) - \sum_{H \in \cG_p^1} \frac{t_H}{n_{H,Q}} p_QF_*(F^*(S(R\G_G)(a))) = 0, $$
and hence $\supp(b) < \supp(a)$. By our inductive assumption, there exist rational numbers $\{z_K\}$ such that 
$$b = \sum_{K \in \cG_p^1}   z_K \Ind_K^G(\Res_K^G(b)). $$
By substituting  the formula defining $b$ into this expression, we obtain terms of the form
$$(\Ind_K^G \circ \Res_K^G\circ \Ind_H^G\circ \Res_H^G)(a)$$
for $p$-hypoelementary subgroups $H$ and $K$.  
However, we can use the Mackey double coset formula to express $\Res_K^G\circ \Ind_H^G$ as a sum of terms of the form
 $$\Ind_{\, \leftexp{g}{H}\cap K}^K\circ\,  c_g\circ  \Res_{H\cap K^g}^H. $$
 Since these terms will be applied to $\Res_H(a)$, and conjugation acts as the identity on $K_0(R\G_G)$, the internal conjugations can be omitted. 
We have now obtained the desired result
$$a = \sum_{H \in \cG_p^1} r_H \Ind_H^G(\Res_H^G(a)),$$
for any $a\in K_0(R\G_G)\cong K_0(R\cE_G)$. 
 Note that when this formula is applied to an element in the Burnside quotient Green ring $\cA_K$, it says that $\cA_K$ is rationally generated by induction from the $p$-hypoelementary subgroups of $G$.
\end{proof}

\section{Oliver's actions on finite acyclic complexes}
\label{sect: Oliver}

In this section, let $R= \Zphat$ or $R=\cy{p}$, for some prime $p$. We prove a result about the finiteness obstruction
$\widetilde\sigma(\CC) \in \widetilde K_0(\RG_G)$
of  a chain complex $\CC$ over the orbit category (with respect to the family of $p$-subgroups of $G$), which is weakly homology equivalent to a finite
projective chain complex. This follows from a more direct result
about modules over the orbit category which have finite projective
resolutions. As an application of these observations, we give an
alternative approach to R.~Oliver's constructions of finite mod-$p$
acyclic complexes.

Given a finite 
$G$-CW-complex $X$, there is an associated finite  cellular chain complex
$$ \uC{X} \colon\quad  0 \to \uR{X_n}\to \cdots  \to \uR{X_1}  \to \uR{X_0}
\to 0 $$ of $\RG _G$-modules, where $X_i$ denotes the set of
$i$-cells in $X$. An $\RG_G$-module of the form $\uR{G/H}$ is not
projective in general, but  it has always a finite projective
resolution (see Example \ref{ex: projective length}).  
 Recall that a \emph{weak homology equivalence} between chain complexes over $\RG_G$ is a chain map  inducing an isomorphism on homology (see \cite[\S 11]{lueck3}).
\begin{lemma}\label{weak homology} The complex $\uC{X}$
is weakly homology equivalent to a finite projective complex $\PP$.
\end{lemma}
\begin{proof} For each $k\geq 0$, we have the $k$-skeleton $X^{(k)}$ of $X$, which is a $G$-CW subcomplex, and a short exact sequence
$$0 \to \uC{X^{(k-1)}} \to \uC{X^{(k)}} \to \DD^{(k)} \to 0$$
of $\RG_G$-module chain complexes, for $k \geq 1$. The relative cellular complex $$ \DD^{(k)} =\uC{(X^{(k)}, X^{(k-1)})} = \uR{X_{k}},$$
 where we regard the module $ \uR{X_{k}}$ as a chain complex concentrated in degree $k$.
For each $k\geq 0$, we pick a finite projective resolution $f^{(k)}\colon \PP^{(k)} \to \uR{X_{k}}$, and regard  $\PP^{(k)}$ as a chain complex starting in degree $k$. The map $f^{(k)}$ then gives a weak homology equivalence
$ \PP^{(k)} \to \DD^{(k)}$. By induction on $k$ and standard homological algebra (see \cite[11.2(c)]{lueck3}), we obtain a weak homology equivalence $f\colon \PP \to \uC{X}$ with $\PP = \bigoplus \PP^{(k)}$.
\end{proof}

The obstruction for replacing a weak homology equivalence $f\colon \PP \to \uC{X}$ with a finite free chain complex (in
the same chain homotopy type) is an element 
$$\widetilde \sigma(X) \in \widetilde
K_0 (\RG_G)$$
in the projective class group,  defined as the image of the Euler characteristic
$$\sigma ( X) =\sum  (-1) ^i [P_i]\in 
K_0 (\RG_G). $$
Note that this obstruction is defined for any finite $G$-CW complex $X$, so
it is defined for finite $G$-sets as well (considered as $G$-CW complexes of dimension zero).

By uniqueness of projective resolutions (up to chain homotopy equivalence), the Euler characteristic $\sigma(X)$, and hence the finiteness obstruction $\widetilde \sigma(X)$,  is independent of the choice of projective complex $\PP$ weakly homology equivalent to $\uC{X}$. In particular, the proof of Lemma \ref{weak homology} shows that 
\eqncount
\begin{equation}\label{G set formula}
\sigma (X)=\sum  (-1) ^k \sigma (X_k) \in  K_0 (\RG_G),
\end{equation}
where $\sigma(X_k)$ is (by definition) the Euler characteristic of any finite projective resolution 
$\PP^{(k)}$ for the module $\uR{X_k}$.
The obstruction $\widetilde\sigma(X) = 0$ if and only if there is a finite free chain complex with a weak homology equivalence to  $\uC{X}$.

We now recall a description of the Burnside ring
$B(G)$, due to tom Dieck \cite[p.~239]{tomDieck1}. In this
description, $B(G)$ is the set of equivalence classes of finite
$G$-CW complexes,  with $X \sim Y$ if and only if $\chi(X^H) =
\chi(Y^H)$ for all subgroups $H \leq G$. The addition is disjoint union and the multiplication is Cartesian product.  The additive identity is
the empty set, and the additive inverse $-[X]$ is represented by $Z
\times X $, for any finite complex $Z$ with $\chi(Z) = -1$ and
trivial $G$-action.

If $X$ is a finite $G$-CW complex, and $\{X_k\}$ denotes the finite $G$-sets of $k$-cells, then the relation
$$[X] =\sum  (-1)^k [X_k] \in B(G)$$
follows immediately from the definition above.  Now this relation and formula (\ref{G set formula}) shows that  $\sigma (X)=\sigma (Y) \in K_0(\RG_G)$ whenever $\chi(X^H)=\chi(Y^H)$, for all subgroups $H\leq G$.

 The main
result of this section is the following improvement:

\begin{theorem}
\label{thm:sigma} Let $X$ and $Y$ be two $G$-CW-complexes such that
$\chi (X^H )=\chi(Y^H)$ for every $p$-hypoelementary subgroup $H$
in $G$, then $\sigma (X)=\sigma (Y) \in K_0(\RG_G)$.
\end{theorem}

As an application, we have a useful embedding result:
\begin{corollary}\label{cor: Oliver} Let $X$ be a finite $G$-CW complex
with the property that $\chi (X^H )=1$ for every $H \in\mathcal{G}
_p ^1 $. Then  there exists a finite $G$-CW complex $Y$ including
$X$ as a subcomplex such that
\begin{enumerate}
\item $Y \backslash X$ only has cells with prime power stabilizers.
\item $Y^K$ is mod $p$ acyclic for every $p$-subgroup $K$.
\end{enumerate}
\end{corollary}

\begin{proof}
Let $R=\cy{p}$. By Theorem \ref{thm:sigma}, $\sigma
(X)=\sigma (pt)$.  By attaching orbits of  cells with stabilizers $Q\in
\cF$, we can also assume that the chain complex $\CC:=\uC{X}$ of the
$G$-CW-complex $X$ is $n$-dimensional, $(n-1)$-connected for $n$ large,  and has a single nontrivial homology $H_n(\CC) =M$ in positive dimensions. This process does not change the finiteness obstruction, so we have $\widetilde\sigma
(X)=\widetilde\sigma (pt)$. 

Since $H_0(\uC{X}) = \un{R}$ has a finite projective resolution, the exact sequence
$$0 \to M \to  C_{n} \to C_{n-1} \to \dots \to C_0 \to \un{R} \to 0$$
implies that $\Ext_{\RG_G}^k(M, N) = 0$, for  all $\RG_G$-modules $N$, if $k$ is sufficiently large. Hence
the $\RG_G$-module $M$ also has a
finite projective resolution and we let $\chi(M) \in K_0(\RG_G)$ denote the Euler characteristic of any such resolution, as in the proof of Lemma \ref{lem: green ring}. But $\sigma(X) = (-1)^n\chi(M) + \chi(\un{R})$, by \cite[11.9]{lueck3}, and $\sigma(pt) = \chi(\un{R})$. Hence the relation $\widetilde\sigma(X) =\widetilde \sigma(pt)$ implies that $\widetilde\chi (M)=0 \in \widetilde K_0(\RG_G)$,
 implying that $M$ has a finite free resolution over $\RG_G$.  This
shows that we can add more cells with stabilizers $Q\in \cF$ to
kill the remaining homology on $X$ and obtain a mod $p$ acyclic
complex satisfying the above properties.
\end{proof}

Before giving the proof of Theorem \ref{thm:sigma}, we need some preparation.
Recall that there is map called the
\emph{linearization map} from $B(G)$ to the Green ring $A(RG)$.
The linearization
map $$\Lin \colon B(G) \to A(RG)$$ is defined as the linear extension of the
assignment $[X] \to [RX]$ where $RX$ denotes the permutation module
with basis given by  a finite $G$-set $X$. The linearization map is determined as follows:

\begin{lemma}[Conlon]\label{lem: Oliver}
For a $G$-CW complex $X$, the class $\Lin _R ([X])=0$ if and only if
$\chi (X^H)=0$ for every subgroup $H \in  \cG_p^1$.
\end{lemma}

\begin{proof} This is due to Conlon (see \cite{conlon1},
or \cite[Theorem 3.5.5]{bouc6}). The ``if" direction is a special case of  \cite[Theorem 7]{dress1}. The ``only if" direction is the statement that the linearization map $B(H) \to A(RH)$ is injective for all $H \in \cG_p^1$ (this also holds for $R=\bZ$ by \cite[Prop.~9.6]{dress1}).
\end{proof}

Note that to prove Theorem \ref{thm:sigma}, it is enough to
prove it for $G$-sets $X$ and $Y$ satisfying the property that
$|X^H|=|Y^H|$ for all $H \in \cG_p^1$. 
By Conlon's theorem,
two such $G$-sets will then have isomorphic permutation modules $RX \cong
RY$. 

\begin{remark}
If $Q\triangleleft H$, for some $p$-subgroup $Q$, then $H/Q \in \cG_p^1$ if and only if $H \in \cG_p^1$. 
We may apply this remark to the $N_G(Q)/Q$-sets $X^Q$ and $Y^Q$. By Conlon's Theorem, the permutation modules $R[X^Q]$ and $R[Y^Q]$ will be isomorphic
as $R[N_G (Q)/Q]$-modules, for every $p$-subgroup $Q$, since 
$|(X^Q)^{H/Q}| = |X^H|$ for all $H/Q \leq N_G(Q)/Q$ with $H/Q \in \cG_P^1$.
\end{remark}

\begin{proof}[The proof of Theorem \ref{thm:sigma}]
We are considering modules over the orbit category $\G_G$ relative to the family
 $\cF$ of all $p$-subgroups in $G$. If $X$ and $Y$ are finite $G$-sets such that
 $RX \cong RY$ as $RG$-modules, then we wish to show that $\sigma(X) = \sigma(Y)$. The argument will proceed in the following two steps:
 \begin{enumerate}
 \item If $G$ is $p$-hypoelementary, and $RX \cong RY$ as $RG$-modules, then we will show that $\uR{X} \cong \uR{Y}$ as $\RG_G$-modules. 
  \item We reduce to $p$-hypoelementary groups by applying Corollary \ref{cor: hypo detection}.
 \end{enumerate}
 
 To establish step (i) we now assume that $G \in \cG_p^1$. Since any  subgroup  of a $p$-hypoelementary group is also $p$-hypoelementary,  we see that $|X^H| = |Y^H|$ for all  $H \leq G$ by Lemma \ref{lem: Oliver}. This shows that $X \cong Y$ as  $G$-sets, and finishes step (i).
 
 For any finite group $G$,  we conclude by step (i) that $\Res^G_H(\uR{X}) \cong \Res^G_H(\uR{Y})$, for all $H \in \cG_p^1$, and therefore $\Res^G_H(\sigma(X)) = \Res^G_H(\sigma(Y))$, for all $H \in \cG_p^1$. By Corollary \ref{cor: hypo detection}, we have  $\sigma(X) = \sigma(Y) \in  K_0(\RG_G)$.
\end{proof}

We remark that step (i) above only holds if $G$ is $p$-hypoelementary.  In general, given two $G$-sets $X$ and $Y$ such that $RX
\cong RY$ as $RG$-modules, we can not conclude that $\uR{X} \cong \uR{Y}$
as $\RG_G$-modules, even though $R[X^Q] \cong R[Y^Q]$ as
$R[N_G(Q)/Q]$-modules for every $Q \in \cF$.  In other words, the Dress detection result (Corollary \ref{cor: hypo detection}) does not extend to $A(\RG_G)$.  Here is an explicit example.

\begin{example}\label{ex:counter-example}
Let $G=S_3$, $R=\cy{2}$ and $\cF$ be the family of all $2$-subgroups
in $G$. Let $X=[G/1]+2[G/G]$ and $Y=2[G/C_2]+[G/C_3]$. Except for $G$,
all subgroups of $G$ are $2$-hypoelementary. It is easy to see that
$|X^K|=|Y^K|$ for all $K\leq G$ and $K\neq G$. So, $RX \cong RY$ as
$RG$-modules and $\sigma(X) = \sigma(Y)$ by Theorem \ref{thm:sigma}. Note that the modules $\uR{G/1}$, $\uR{G/C_2}$, and
$\uR{G/C_3}$ are all projective as $\RG_G$-modules, but $\uR{G/G}$ is
not since $G$ does not have a normal Sylow $2$-subgroup (see \cite[Lemma 2.5]{symonds2}). Therefore, we can
not have an isomorphism $\uR{X} \cong \uR{Y}$, otherwise $\uR{G/G}$
would be projective.
\end{example}

As an application of Theorem \ref{thm:sigma}, we will prove the
following theorem of Oliver which is the key result in
\cite{oliver0}.

\begin{theorem}[{Oliver \cite[Theorem 1]{oliver0}}]\label{Oliver's theorem}
Let $G$ be a finite group not of $p$-power order, and $\varphi$ a
mod $p$ resolving function for $G$. Then for any finite complex $F$
with $\chi (F)=1+\varphi (G)$, $F$ is the fixed-point set of an
action of $G$ on some finite $\cy{p} $-acyclic complex.
\end{theorem}

A mod $p$ resolving function is defined by Oliver in the following way:

\begin{definition} A {\em mod $p$ resolving function} for $G$ is a
super class function $ \varphi$ satisfying the following properties:
\begin{enumerate}
\item $|N_G (K)/K |$ divides $\varphi (K) $ for all $K \leq G$.
\item For any $K\leq G$ such that $K \in \mathcal{G} _p ^1 $, we have $\sum _{K
\leq L } \varphi (L) =0$.
\end{enumerate}
\end{definition}

We will give alternate description of mod $p$ resolving functions. Note that there is a commutative diagram
$$\xymatrix{0 \ar[r]&  B(G) \ar[r]^{\rho} \ar@{=}[d] &  C(G) \ar[r]^{\psi}
\ar@<2pt>[d]^{\theta} & \Obs(G) \ar[r] \ar@{=}[d]  & 0 \\ 0 \ar[r] &
B(G) \ar[r]^{\eta} & C (G) \ar[r]^{\gamma}
\ar@<2pt>[u]^{\theta^{-1}} & \Obs(G) \ar[r] & 0 }$$ where $B(G)$
denotes the Burnside ring of finite $G$-sets, $C(G)$ denote the
group of super class function, and $$\Obs (G)=\bigoplus _{K \leq _G
G} \bZ / |W_G (K)| \bZ. $$ The maps in the diagram are defined as
follows: the map $\rho$ is the \emph{mark} homomorphism
\cite{dress3} defined by $\rho (G/K)(L)=|(G/K)^L |$, and $\eta$ is
defined by $ \eta ([G/K]) (L)=|W_G (K)|$ if $K$ and $L$ are
conjugate to each other, and zero otherwise. The homomorphism
$\gamma$ is defined as the direct sum of the mod $|W_G (K)|$ reductions, and $\psi= \gamma\circ \theta$. The map
$\theta$ is an invertible transformation such that
$$ \theta (f) (K) =\sum _{K\leq L } \mu (K,L) f(L) \quad \quad {\rm and}
\quad \quad  \theta ^{-1} (f) (K) =\sum _{K \leq L } f(L).$$ Here
$\mu (K,L)$ denotes the M\" obius function for the  poset of
subgroups of $G$. More details about this diagram can be found in
\cite{coskun-yalcin1}. We have the following observation:

\begin{lemma}\label{lem: super}
 A super class function $\varphi$ is a mod $p$
resolving function if and only if $\theta ^{-1} (\varphi ) $ is in
the image of $\rho$ and $\theta ^{-1} (\varphi ) (K)=0$ for all $K
\in  \cG_p^1 $.
\end{lemma}

\begin{proof} This follows form the above commuting diagram and
from the definition of mod $p$ resolving functions.
\end{proof}
\begin{remark} Given $F$, and any group $G$ not of $p$-power order, Oliver concludes in \cite[Corollary, p.~162]{oliver0} that there exists a finite $\cy{p}$-acyclic $G$-CW complex with $X^G=F$ if and only if $$\chi(F) \equiv 1 \text{\ mod\ }m_p(G),$$
 where $m_p(G)$ is the greatest common divisor of the integers 
$\{ \varphi(G)\}$ over all mod $p$ resolving functions for $G$.
The existence of mod $p$
resolving functions for $G$ is completely analyzed by Oliver  in \cite[Theorem 4]{oliver0}, which gives the explicit characterization: (i) $m_p(G)=0$  if and only if $G \in \cG_p^1$, (ii) $m_p=1$ if $G\notin \cG_p$, and (iii) $m_p$ is the product of the distinct primes $q$ such that $G \in \cG_p^q$, $q >1$. 
\end{remark}

\begin{proof}[The proof of Theorem \ref{Oliver's
theorem}]
Let $\varphi$ be  a mod $p$ resolving function $\varphi$, and $F$ a finite complex such that $\chi (F)=1+\varphi(G)$. Then by Lemma \ref{lem: super}, the super class function $f=\theta ^{-1} (\varphi)$ in the image of $\rho$, and $\varphi(G) = f(G)$ so that $\chi (F)=1+f(G)$. Notice that $1+f$ is also in the image of  $\rho \colon B(G) \to C(G)$, since $\rho([G/G])(K) = 1$ for all $K\leq G$. 

From tom Dieck's description  \cite[p.~239]{tomDieck1} of $B(G)$, there 
exists a finite $G$-CW complex $X$ with the properties:
\begin{enumerate}
\item $\chi (X^K)=1+f(K)$ for
every $K \leq G$,
\item $\chi(X^H) = 1$ for all $H \in \cG_p^1$, and
\item $X^G=F$
\end{enumerate} 
Property (ii) follows from the definition of $\theta$, since $\varphi = \theta(f)$ is a mod $p$ resolving function.
To obtain property (iii),  start with any finite $G$-CW complex $X_1$ satisfying property (i) and let $X_0 = X_1 \setminus U$, where $U$ is an open
$G$-invariant regular neighbourhood of $X_1^G$, 
obtained by an equivariant triangulation
of $X_1$. Let $X = X_0
\disjointunion F$. Note that $X_0$ and therefore $X$ has the
structure of a finite $G$-CW complex.  Then
$\chi(X_0^K) = \chi(X_1^K) - \chi(X_1^G)$, for all $K \leq G$. It
follows that $\chi(X^K) = \chi(X_1^K)$, for all $K\leq G$, and hence
$[X] = [X_1]\in B(G)$.

By Corollary \ref{cor: Oliver}, there  exists a
mod $p$-acyclic $G$-CW complex $Y$, containing $X$ as a subcomplex,
such that $Y \backslash X$ only has cells with prime power
stabilizers. Since $G$ is not of $p$-power order, it follows that $Y^G=X^G=F$, and this completes the proof.
\end{proof}

\begin{remark}
Oliver \cite[\S 3]{oliver0} also determined which finite complexes $F$ appear as the fixed-point set $X^G$, for finite contractible $G$-CW complexes $X$. 
An \emph{integral resolving function for $G$} is a super class function  in $C(G)$ which is a mod $p$ resolving function for all primes $p$. The set of integral resolving functions forms a group, and $m(G)$ is defined as the greatest common divisor of the integers $\varphi(G)$ over all integral resolving functions. 

Assume that $G$ is not of prime power order. Given an integral resolving function $\varphi$ for $G$, and a non-empty finite complex $F$ such that $\chi(F) = 1 + \varphi(G)$, there exists a finite $G$-CW complex $X$ such that $\chi(X^H) = 1$, for all $H \in \cG_p^1$ and all primes $p$, and with $X^G = F$ (as in the proof of Theorem \ref{Oliver's
theorem} above).

Let $\G_G$ denote the orbit category of $G$ with respect to the family $\cF$ of all $p$ subgroups, for all primes $p$. By attaching orbits of  cells with stabilizers $Q\in
\cF$, we can also assume that the chain complex $\CC:=\uC{X}$ of the
$G$-CW-complex $X$ is $n$-dimensional, $(n-1)$-connected for $n$ large, with $H_0(\CC) = \un{\bZ}$ and has a single nontrivial homology $H_n(\CC) =M$ in positive dimensions.

For each prime $p$, the homology modules $H_i(\CC)\otimes \Zphat$ admit finite projective resolutions over $\Zphat\G_G$, so by \cite[Prop.~3.11]{hpy1} the homology modules $H_i(\CC)$  admit finite projective resolutions over $\ZG_G$. Therefore,   the finiteness obstruction $\sigma(X)$ is defined, and 
$$\widetilde\sigma(X) = (-1)^n[M] + [\un{\bZ}]\in \widetilde K_0(\ZG_G)$$ 
by \cite[11.9]{lueck3}. 
We call $X$ a $G$-\emph{resolution of  $F$}, and define $\gamma_G(F,X) := \widetilde\sigma(X)$, following Oliver \cite[\S 3]{oliver0}. Then define
$$\gamma_G(F) \in  \widetilde K_0(\ZG_G)/\cB(G)$$
 to be the image of $\gamma_G(X,F)$, for any  $G$-resolution $X$ of $F$, where
$$\cB(G) = \{ \gamma_G(pt, X)\vv X \text{\ is a $G$-resolution of\ }F=pt\}.$$
Then $\gamma_G(F)$ is well-defined, as in \cite[Prop.~5]{oliver0}. If $\chi(F) =1$, and $X$ is a $G$-resolution for $F$, then $X/F$ is a $G$-resolution for $(X/F)^G = pt$, and hence $\gamma_G(F) = 0$ whenever $\chi(F) =1$. It follows as in \cite[Theorem 3]{oliver0} that $\chi(F_1) = \chi(F_2)$ implies $\gamma_G(F_1) = \gamma_G(F_2)$. Since
$\gamma_G(F_1 \vee F_2) = \gamma_G(F_1) + \gamma_G(F_2)$, we also have the conclusion of \cite[Corollary 5]{oliver0}. Let $n_G$ denote the greatest common divisor of the integers $\{\chi(F) -1\}$ as $F$ varies over all finite complexes with
$\chi(F) \equiv 1 \Mod m(G)$ and $\gamma_G(F) = 0$. Then $F$ is the fixed point set of a finite contractible $G$-CW complex if and only if $\chi(F) \equiv 1 \Mod n_G$.

 It might be interesting to continue the study of $\cB(G)$ over the orbit category, in analogy with Oliver \cite{oliver4}.
\end{remark}

\section{Acyclic permutation complexes}
\label{sect: Acyclic}

Let $G$ be a discrete group. We say that $X$ is a $G$-complex if $X$
is a CW-complex with a $G$-action on it in a such a way that $G$
permutes the cells in $X$ and if $G$ fixes a cell, then it fixes it
pointwise. Note that a $G$-CW-complex is a $G$-complex and
conversely, every $G$-complex has a $G$-CW-complex structure. For
$G$-complexes, we have the following theorem of tom Dieck (Chapter
II, Proposition 2.7 in \cite{tomDieck2}):

\begin{theorem} If $G$ is a discrete group and $f\colon  X \to Y$ is a
$G$-map between $G$-CW-complexes which induces homotopy equivalences
$X^H \to Y^H $ between the $H$-fixed subspaces for all subgroups $H
\leq G$, then $f$ is itself a $G$-homotopy equivalence.
\end{theorem}

Recently, Kropholler and Wall \cite{kropholler-wall1} gave an
algebraic version of this theorem. To introduce their theorem, we
need to give more definitions.

Let $R$ be a commutative ring and $X$ be a $G$-set. As usual, we
denote by $RX$, the \emph{based} $RG$-permutation module with basis $X$ where $G$
acts by permuting the basis. An $RG$-module homomorphism $f \colon  RX \to
RY$ between two based permutation modules is called {\it admissible} if it
carries the submodule $R[X^H]$ into $R[Y^H]$ for all subgroups $H
\leq G$. A chain complex of based $RG$-permutation modules $$ \CC \colon  \cdots
\to R[X_n] \to R[X_{n-1}] \to \cdots \to R[X_1] \to R[X_0] \to 0 $$
is called a {\it special $G$-complex} if all the boundary maps are
admissible. A chain map $f \colon  \CC \to \DD$ between special
$G$-complexes is a called an admissible $G$-map if for each $i$, the
map $f_i \colon  C_i \to D_i$ is an admissible map.

\begin{theorem}[Kropholler-Wall \cite{kropholler-wall1}]\label{Krop-Wall theorem}
Let $f \colon  \CC \to \DD$ be an admissible $G$-map between special
$G$-complexes.  If $f$ induces a chain homotopy equivalence between
the $H$-fixed subcomplexes for all subgroups $H \leq G$, then $f$ is
itself a chain homotopy equivalence.
\end{theorem}

Here by an $H$-fixed point complex, we mean the subcomplex
$$\to R[X_n ^H] \to R[X_{n-1}^H] \to \cdots \to
R[X_1 ^H ] \to R[X_0 ^H] \to 0.$$ It is clear that when $f \colon
\CC \to \DD$ is an admissible $G$-map, then for each $H \leq G$, it
induces a chain map between fixed point complexes.

Observe that a based permutation $RG$-module $R[X]$ can be considered as
a module over the orbit category in a natural way: let $\G_G = \Or G$
denote the orbit category over all subgroups in $G$. Associated to a
permutation $RG$-module $R[X]$ with an $R$-basis $X$, there is an
$\RG_G$-module $\uR{X}$ which is a free $\RG_G$-module. Note that if $f
\colon RX \to RY$ is admissible, then it induces an $R\G_G$-module map
$f \colon \uR{X} \to \uR{Y}$. Conversely, given a map between free
$R\G_G$-modules $f\colon \uR{X} \to \uR{Y}$, evaluation of $f$ at $1$
gives an admissible map $f(1)\colon  RX\to RY$. This gives a natural
equivalence between the following two categories:

\begin{enumerate}
\item The category of based $RG$-permutation modules and admissible maps.
\item The category of free $\RG_G$-modules and $\RG_G$-module maps.
\end{enumerate}

The equivalence of these categories gives an alternative proof for
Theorem \ref{Krop-Wall theorem} using the orbit category.

\begin{proof} Let $f\colon  \CC \to \DD$ be a admissible $G$-map between special
$G$-complexes. Under the natural equivalence explained above, we can
consider $f$ as a chain map between free chain complexes of
$\RG_G$-modules. The condition that $f$ induces homotopy equivalences
between the $H$-fixed subcomplexes for all $H \leq G$ gives that
$f(H) \colon  \CC (H) \to \DD (H)$ is an homotopy equivalence for
all $H \leq G$. This gives, in particular, that the induced map on
homology $ f_* (H) \colon  H_* (\CC (H) ) \to H_* (\DD (H) )$ is an
isomorphism for all $H \leq G$. But, $H_* (\CC (H))= H_* (\CC )
(H)$, so we get that $f_* \colon  H_* (\CC ) \to H_* (\DD )$ is an
isomorphism of $\RG_G$-modules. Now, by a standard theorem in
homological algebra, this implies that $f \colon  \CC \to \DD$ is a
chain homotopy equivalence as a chain map of $\RG_G$-modules.
Evaluating $f$ at $1$, we get the desired result.
\end{proof}

Our interpretation of the next result will use the following version of
Smith theory:

\begin{theorem}[{Symonds \cite[Corollary 4.5]{symonds3}}]\label{thm: symonds}
Let $G$ be a $p$-group, $\G_G=\Or G$, and $R=\Zphat$ denote the
$p$-adic integers. If $\CC$ is a chain complex of projectives over
$\RG_G$ that is bounded above, such that $\cy{p} \otimes _{\bZ} \CC
(1)$ is exact, then $\CC$ is split exact.
\end{theorem}

\begin{proof} This is a slight generalization of Corollary 4.5 in Symonds
\cite{symonds3} and the proof follows easily from the argument given in \cite{symonds3} (see also Section 6 of 
Bouc \cite{bouc8} for similar results).
\end{proof}

In \cite{kropholler-wall1}, Kropholler and Wall also gave an
alternative proof for a theorem of Bouc \cite{bouc7} about acyclic
simplicial complexes (and extended the statement to special $G$-complexes). 
We will give a proof using the orbit category and Theorem \ref{thm: symonds}.
Recall that, a complex $\CC$ of $RG$ modules is called
acyclic if it has zero homology everywhere except at dimension zero
and $H_0 (\CC)=R$. Also note that a complex of $RG$-modules is
called $G$-split if it admits a chain contraction.

\begin{theorem}[Kropholler-Wall \cite{kropholler-wall1}]\label{Boucthm}
Let $G$ be a finite group and let $\CC$
be a finite dimensional special $\bZ G$-complex. If $\CC$ is
$\bZ$-acyclic, then the augmented chain complex $\widetilde \CC$ is
$G$-split.
\end{theorem}

\begin{proof} The augmented chain complex
$$\widetilde \CC \colon  \ 0 \to C_n \to \cdots \to C_1 \to C_0 \to \bZ \to 0$$
is an exact sequence of $\bZ G$-permutation modules. To show that
$\CC$ is $G$-split, we need to show that the short exact sequences
$$0 \to Z_i \to C_i \to Z_{i-1} \to 0$$ in $\widetilde \CC$
are all split exact sequences of $\bZ G$-modules. Since all the
modules involved are free over $\bZ$, the extension classes $\Ext ^1
_{\bZ G} (Z_{i-1}, Z_i )$ are detected by restriction to the Ext-groups $\Ext ^1 _{\bZ
_p P} (Z_{i-1}, Z_i )$ where $P$ is a Sylow $p$-subgroup of $G$. So,
it is enough to assume that $G=P$ is a $p$-group and show that
$\bZ_p \otimes _{\bZ} \widetilde \CC$ is split.

As before, we can consider the complex $\widetilde \CC$ as a complex
of $\bZ \G_G$-modules where $\G_G =\Or G$. This is a chain complex of
the form
$$\DD : 0 \to \uZ{X_n} \to \cdots \to \uZ{X_1} \to \uZ{X_0}
\to \uZ{G/G} \to 0,$$ where all the modules are free $\bZ
\G_G$-modules. Evaluation of $\DD$ at 1 gives the augmented complex
$\widetilde \CC$. Since $\CC$ is acyclic, the complex $\cy{p}
\otimes _{\bZ} \widetilde \CC  =\cy{p} \otimes _{\bZ} \DD(1)$ is
exact by universal coefficient theorem. So, by Theorem \ref{thm:
symonds}, we obtain that $\Zphat \otimes _{\bZ} \DD$ is split exact,
hence its evaluation at 1, which is the complex $\Zphat \otimes
_{\bZ} \widetilde \CC $, is also split exact.
\end{proof}

\begin{remark} Kropholler and Wall \cite[\S 5]{kropholler-wall1} observed that
Oliver's results on fixed point sets of finite contractible $G$-CW complexes  combined with Theorem \ref{Boucthm} imply a Dress induction statement. Here is a variant of that observation: let $F$ be a finite complex with $\chi(F) = 1 + m_p(G)$, for some prime $p$. Then there is a finite mod $p$ acyclic $G$-CW complex $X$ with $X^G=F$. By the mod $p$ version of Theorem \ref{Boucthm},  the augmented chain complex $\DD$ of $\uC{X}$, for $R = \cy p$, is split over the orbit category and hence its evaluation $\widetilde \CC = \DD(G/1)$ is $G$-split. This gives the relation 
that $m_p(G)\cdot [R]$ is a linear combination in $A(RG)$ of permutation modules
$R[G/H]$, with $H<G$ proper subgroups. This suggests that $m_p(G)$ should be the optimal denominator in the Dress rational hypoelementary induction Theorem
\cite[Theorem 7]{dress1}. 

It should also be pointed out that this implication is circular, since the proof of Oliver's results involves directly or indirectly the same ingredients as Dress's theorem.
\end{remark}

Some other nice applications of Theorem \ref{Boucthm} are given in
\cite{kropholler-wall1}. One of them extends a result of Floyd \cite[Theorem 2.12]{floyd1}.

\begin{theorem} [Theorem 6.1, \cite{kropholler-wall1}]
Let $G$ be a locally finite group and let $X$ be a finite
dimensional acyclic $G$-CW complex. Then, the complex $X/G$ is acyclic.
\end{theorem}
\begin{proof} We outline the steps of the argument given in  \cite{kropholler-wall1}. Since a locally finite group is the directed union of its finite subgroups, it is enough to do the case where $G$ is finite. Then $C(X; \bZ)$ is $\bZ G$-split, implying that the chain complex $C(X; \bZ)\otimes_{\bZ G} \bZ$ is also acyclic by Theorem \ref{Boucthm}. But this complex is isomorphic to the chain complex of $X/G$.
\end{proof}

\providecommand{\bysame}{\leavevmode\hbox to3em{\hrulefill}\thinspace}
\providecommand{\MR}{\relax\ifhmode\unskip\space\fi MR }
\providecommand{\MRhref}[2]{%
  \href{http://www.ams.org/mathscinet-getitem?mr=#1}{#2}
}
\providecommand{\href}[2]{#2}

\end{document}